\documentclass[11pt]{article}
\usepackage{amsmath, amsfonts, amssymb, amsthm,epsfig}
\usepackage{natbib}
\usepackage{setspace}
\usepackage[ps2pdf,linktocpage,bookmarks,bookmarksnumbered,bookmarksopen]{hyperref}
\hypersetup{%
   pdftitle   ={The coding complexity of L\'{e}vy processes},
   pdfsubject ={Manuscript},
   pdfauthor  ={Frank Aurzada and Steffen Dereich},
   pdfkeywords={High-resolution quantization; distortion-rate function; complexity; L\'{e}vy processes}
      }

\let\BFseries\bfseries\def\bfseries{\BFseries\mathversion{bold}} % formulas in headings bold
\DeclareMathSymbol{\leqslant}{\mathalpha}{AMSa}{"36} % nicer `smaller or equal'
\DeclareMathSymbol{\geqslant}{\mathalpha}{AMSa}{"3E} % nicer `larger or equal'
\DeclareMathSymbol{\eset}{\mathalpha}{AMSb}{"3F}     % nicer `emptyset'
\renewcommand{\leq}{\;\leqslant\;}                   % redef. of < or =
\renewcommand{\geq}{\;\geqslant\;}                   % redef. of > or =
\newcommand{\inftwo}[2]{\inf_{\substack{#1 \\ #2}}} % inf with 2 lines
\newcommand{\range}{\mathrm{\,range\,}}
\newcommand{\ind}{1\hspace{-0.098cm}\mathrm{l}}

\newcommand{\cC}{\mathcal{C}}

\newcommand{\cE}{\mathcal{E}}
\newcommand{\cF}{\mathcal{F}}

\newcommand{\cT}{\mathcal{T}}
\newcommand{\cO}{\mathcal{O}}

\newcommand{\cU}{\mathcal{U}}

\newcommand{\om}{\omega}

\newcommand{\ep}{\epsilon}
\newcommand{\eps}{\varepsilon}
\newcommand{\lam}{\lambda}
\newcommand{\sig}{\sigma}
\newcommand{\vphi}{\varphi}

\newcommand{\lfl}{\lfloor}
\newcommand{\rfl}{\rfloor}

\newcommand{\dto}{\downarrow}

\newcommand{\sgn}{\mathrm{sgn}}

\newcommand{\ID}{\mathbb{D}}
\newcommand{\IP}{\mathbb{P}}

\newcommand{\IN}{\mathbb{N}}

\newcommand{\IZ}{\mathbb{Z}}

\newcommand{\IR}{\mathbb{R}}
\newcommand{\IE}{\mathbb{E}}
\newcommand{\iN}{\in\IN}

\newcommand{\iR}{\in\IR}

\newcommand{\be}{\begin{eqnarray*}}
\newcommand{\ee}{\end{eqnarray*}}
\newcommand{\ben}{\begin{eqnarray}}
\newcommand{\een}{\end{eqnarray}}

\theoremstyle{plain}
\newtheorem{theo}{Theorem}[section]
\newtheorem{lemma}[theo]{Lemma}
\newtheorem{propo}[theo]{Proposition}

\theoremstyle{definition}

\newtheorem{remark}[theo]{Remark}
\newtheorem{example}[theo]{Example}

\renewenvironment{proof}[1][] {{\bf Proof#1.} }{\hspace*{\fill}$\square$\medskip\par}

\def\P{{\bf {\mathbb{P}}}}
\newcommand{\pr}[1]{\P\left(#1\right)}
\def\ep{{\varepsilon}}
\def\E{\mathbb{E}}
\newcommand{\norm}[1]{\left\|#1\right\|}
\newcommand{\indi}[1]{\,\ind_{\{#1\}}}
\def\TT{\mathcal{T}}\def\R{\IR}
\def\N{\mathbb{N}}
\def\gri{g}                                 % nearest point in the $\ep \IZ$ grid
\begin{document}
\vglue20pt \centerline{\huge\bf The coding complexity of}
\medskip

\centerline{\huge\bf  L\'evy processes}

\bigskip
\bigskip

\centerline{by}
\bigskip
\medskip

\centerline{{\Large Frank Aurzada and Steffen Dereich}}
\bigskip

\begin{center}\it
Institut f\"ur Mathematik, MA 7-5, Fakult\"at II \\
Technische Universit\"at Berlin
\\ Stra\ss e des 17.\ Juni 136 \\ 10623 Berlin\\
\{aurzada,dereich\}@math.tu-berlin.de
\end{center}

\bigskip
\begin{center}\today
\end{center}
\bigskip
\bigskip
\bigskip

{\leftskip=1truecm \rightskip=1truecm \baselineskip=15pt \small

 \noindent{\slshape\bfseries Summary.} We investigate the high resolution  coding problem for general real-valued L\'evy processes under $L^p[0,1]$-norm distortion. Tight asymptotic formulas are found under mild regularity assumptions.
\bigskip


\bigskip


}
%%%%%% End of narrower

\section{Introduction and Results}\label{intro}
\subsection{Motivation and Notation}
In this article, we study the coding problem for real-valued L\'evy processes $X$
(\emph{original}) under $L^p[0,1]$-norm distortion for some fixed $p\in[1,\infty)$. Here we think of $X$ being a $\ID[0,\infty)$-valued process, where $\ID[0,\infty)$ denotes the space of c\`adl\`ag functions endowed with the Skorohod topology.
%For $t>0$ we let $\IC[0,t]$ denote the set of real-valued
%continuous functions defined on $[0,t]$, and let $\|\cdot\|_{[0,t]}$
%denote the corresponding supremum norm that is
%$\|f\|_{[0,t]}=\sup_{u\in[0,t]} |f(u)|$. Mostly we will consider
%$\|\cdot\|=\|\cdot\|_{[0,1]}$.
We shall denote by $\|\cdot\|$ the standard $L^p[0,1]$-norm.

Let $0<s\le \infty$. The objective is now to find a c\`adl\`ag real-valued process $\hat X$  (\emph{reconstruction} or \emph{approximation})
that minimizes the error criterion
\begin{align}\label{eq515-1}
\bigl\| \| X-\hat X \| \bigr \| _{L^s(\IP)}= \begin{cases} \IE [\|X-\hat X\|^s]^{1/s} &\text{ if }s<\infty\\
\mathrm {ess\,sup}\, \|X-\hat X\| & \text{ if }s=\infty
                                  \end{cases}
\end{align}
under a given complexity constraint on the approximating random variable $\hat X$. We will work with the following three  complexity constraints that have been originally suggested by Kolmogorov \cite{Kol68}: for $r\geq 0$,
\begin{itemize}
\item $\log |\range (\hat X)|\le r$ (\emph{quantization constraint})
\item $H(\hat X)\le r$, where $H$ denotes the entropy of $\hat X$ (\emph{entropy constraint})
\item $I(X;\hat X)\le r$, where $I$ denotes the Shannon mutual information of $X$ and $\hat X$ (\emph{mutual information constraint}).
\end{itemize}
We will work with the following standard notation for entropy and mutual information:
$$
H(\hat X)=\begin{cases} -\sum_{x} p_x \log p_x&\text{ if }\hat X\text{ is discrete with probability weights }(p_x)\\
\infty &\text{ otherwise}
          \end{cases}
$$
and
$$
I(X,\hat X)= \begin{cases}
\int \frac{d\IP_{X,\hat X}}{d\IP_X\otimes \IP_{\hat X}} \,d\IP_{X,\hat X} & \text{ if } \IP_{X,\hat X}\ll \IP_X\otimes \IP_{\hat X}\\
\infty & \text{ otherwise.}
             \end{cases}
$$
Here, $\IP_Z$ denotes the distribution function of a random variable $Z$.

When considering the quantization constraint, we get the following minimal value
$$
D^{(q)}(r,s):=\inf\bigl\{ \bigl\|\|X-\hat X\|\bigr\|_{L^s(\IP)} : \log |\range (\hat X)|\le r\bigr\},
$$
which we call the (minimal) \emph{quantization error} for the \emph{rate} $r\ge 0$ and the \emph{moment} $s$. Analogously, we denote by $D^{(e)}(r,s)$ and $D(r,s)$ the minimal values under the entropy- and mutual information constraint, respectively.  $D^{(e)}$ and $D$ will be called entropy coding error and distortion rate function, respectively. We have $D\leq D^{(e)} \leq D^{(q)}$, for any random variable.

The quantization constraint naturally appears, when  coding the signal $X$ under a strict bit-length constraint. The entropy constraint corresponds to an average bit-length constraint and  the mutual information constraint gains its importance from Shannon's celebrated source coding theorem. In this article we will not consider the run time behaviour of our coding schemes. However, we think that the approximation schemes (provided later in the article) have implementations with reasonable runtime behaviour.
Strictly speaking, the quantities $D^{(e)}$ and $D$ depend on the probability space. However, this dependence has no effect on our results.

The objective of the article is
\begin{itemize}
\item to provide efficient coding strategies for general L\'evy processes that are parameterized by three parameters and that are robust under a mismatch on the L\'evy measure  and
\item to complement the estimates by appropriate lower bounds that show weak optimality of our scheme for most cases.
\end{itemize}

In the article, $X=(X_t)_{t\in[0,\infty)}$ denotes a L\'evy process in the Skorohod space $\ID[0,\infty)$, that is a process starting in $0$ with independent and stationary increments. Due to the L\'evy Khintchine formula, the characteristic function of each marginal $X_t$ ($t\in[0,1]$) admits a representation
\begin{align}\label{eq0117-1}
\IE e^{iuX_t}=e^{-t\psi(u)},
\end{align}
where
$$
\psi(u)=\frac{\sig^2}2 u^2 + i b u+\int_{\IR\backslash\{0\}} (1-e^{iux}+\ind_{\{|x|\le 1\}} iux)\,\nu(dx)
$$
for parameters
$\sig^2\in[0,\infty)$, $b\iR$, and a positive measure $\nu$ on $\IR\backslash\{0\}$ with
\begin{equation} \label{eqn:lm}
\int_{\IR\backslash\{0\}} 1\wedge x^2 \,\nu(dx)<\infty.
\end{equation}
On the other hand, for a given triplet $(\nu,\sig^2,b)$ there exists a L\'evy process $X$ such that (\ref{eq0117-1}) is valid, moreover the distribution of a L\'evy process $X$ is uniquely characterized by the latter triplet.
We will call the corresponding process an $(\nu,\sig^2,b)$-L\'evy process.

If (\ref{eq0117-1}) is true for
$$
\psi(u)=\frac{\sig^2}2 u^2+\int_{\IR\backslash\{0\}} (1-e^{iux}+iux) \,\nu(dx),
$$
then we will call $X$ a $(\nu,\sig^2)$-L\'evy martingale. Note that such a representation implies that $\int |x|\wedge x^2 \,\nu(dx)$ is finite and that the L\'evy process $X$ is a martingale in the usual sense.

After stating our main results in Section~\ref{sec:results}, we shall list some important examples in Section~\ref{sec:exa}. Then Section~\ref{sec:upper} is devoted to the analysis of a particular coding scheme.
The coding strategy of interest will be a measurable function
$$
\Theta=\Theta_{\eps,b,m}:\ID[0,1)\to\ID[0,1)
$$
depending on three parameters $\eps>0$, $b\iR$ and $m>0$.
The parameter $\eps$ will be responsible for the quality of the reconstruction, in the sense that lower $\eps$ correspond to lower approximation errors. The parameters $b$ and $m$ have to be adjusted to $\eps$ and certain quantities relying on the L\'evy measure. Namely, the coding scheme presented below works in a weakly optimal way (in the sense of both quantization constraint and entropy constraint coding error) if $m=m(\eps)$ is the mean number of jumps to be encoded and $b=b(\eps)$ is a drift compensation term. If the generating triplet of the L\'evy process is given, these parameters are explicitly available for computation. If the generating triplet is not known, these values can be estimated from the data.

In Section~\ref{sec:lowerbound}, we derive lower bounds for the above coding problems. Together, these results show that the provided coding scheme is weakly optimal in many cases.

Throughout, we use the following notation for strong and weak asymptotics.
For two functions $f$ and $g$, $f(x) \sim g(x)$, as $x\to 0$, means that
$f(x)/g(x) \to 1$, as $x\to 0$. On the other hand, we use the notation
$f(x) \lesssim g(x)$, as $x\to 0$, if  $\lim_{x\to 0} f(x)/g(x)\leq 1$.
We also write $g(x)\gtrsim f(x)$ in this case. Furthermore, we write
$f(x) \approx g(x)$, as $x\to 0$, if $0<\liminf_{x\to 0} f(x)/g(x)\leq \limsup_{x\to 0} f(x)/g(x) < \infty$.

\subsection{Results} \label{sec:results}
The crucial quantities describing the coding complexity of L\'evy processes are
$$F_1(\ep):= \ep^{-2} \left( \sigma^2+ \int x^2\wedge \eps^2 \, \nu ( d x )\right) \qquad\text{and}\qquad F_2(\ep) := \int_{[-\ep,\ep]^c} \log \left( |x| / \ep \right) \, \nu ( d x).$$
Furthermore, we shall use $F(\ep):= F_1(\eps)+F_2(\eps)$. The function integrated by the L\'{e}vy measure is visualised in Figure~\ref{fig:testfunct03}.
Note that (\ref{eqn:lm}) does not ensure the finiteness of $F_2$ and that $F_2$ is either finite or infinite for all $\eps>0$.

%However, for the quantization error it is implied by the necessary condition $\int |x|^s \nu ( d x) <\infty$, which is needed to ensure that the quantization error tends to zero at all (cf.\ Remark~\ref{rem:conda}).

%Finally let us remark that throughout this paper all logarithms are taken w.r.t.\ base $2$, otherwise, we write $\ln=\log_e$.

\begin{figure}
    \centering
        \epsfig{file=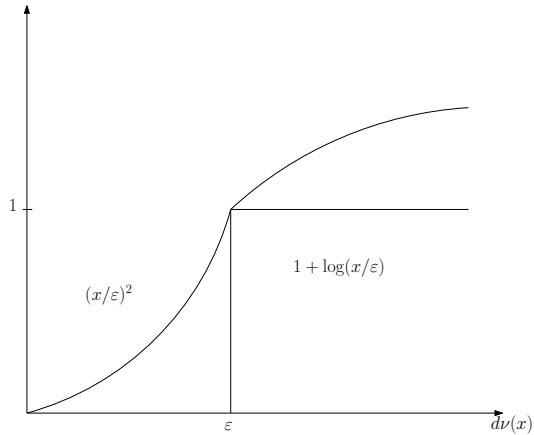,scale=0.4}
    \caption{Visualization of the function $F$}
    \label{fig:testfunct03}
\end{figure}

We are now in a position to state the main results of the article. Let us start with the entropy coding error.

\begin{theo} \label{thm:upperentropy} There exist constants $c_1=c_1(p)>0$ and $c_2>0$ such that, for arbitrary L\'evy processes with finite $F_2$, any $s>0$, and all $\ep>0$,
$$D^{(e)}(\, c_1 F(\ep)  ,s) \leq c_2 \,\ep.$$
\end{theo}

Similarly to the entropy coding error, we obtain the upper bound for the quantization error.
\begin{theo} \label{thm:upperquant} Assume that there is a $q>s$ such that
\begin{itemize}
\item[(a)] $\E\norm{X}^q < \infty$,
\item[(b)] for some $\mu>0$, \begin{equation}\limsup_{\eps\to 0} \frac{\int_{|x|>\eps} ( |x|/\eps)^\mu \, \nu(d x)}{\nu([-\eps,\eps]^c)} <\infty. \label{eqn:addcond}\end{equation}
% \item[(*)] $\limsup_{\ep\to 0} \frac{\log 1/\ep}{F(\ep)} = 0$, and
\end{itemize}
Then there exist a constant $c_1=c_1(p,\nu)>0$ and a universal constant $c_2>0$ such that, for all $0<\eps<\eps_0=\eps_0(\nu,s,p)$,
$$D^{(q)}(\,c_1 F(\ep) ,s) \leq c_2 \,\ep.$$
\end{theo}

In the proofs of the upper bounds we only need to consider the case where $F_2$.
Indeed, in the second theorem, assumption (a) implies the finiteness of $F_2$.

\begin{remark} Let us comment on the conditions in Theorem~\ref{thm:upperquant}: Condition (a) is natural, though one could soften it by the use of Orlicz norms.  Moreover, condition (b) is needed to guarantee that typical realizations of the L\'evy process dominate the quantization complexity of the process (see equation (\ref{eqn:concentr4})). Essentially, (b) does not hold if the L\'evy measure is finite or if $\nu([-\eps,\eps]^c)$ does not grow to infinity fast enough, when $\eps$ tends to zero.

With given L\'{e}vy measure, it is usually easy to verify conditions (a) and (b), cf.\ Remarks~\ref{rem:conda} and~\ref{rem:condbc} below.
\end{remark}

\begin{remark} Another approach for the quantization of L\'{e}vy processes is taken in \cite{LP}. There, linear quantizers are constructed, and a relation of quantization to the path regularity of processes is outlined. However, as observed by Creutzig \cite{creutzig2006}, linear approximations are not optimal whenever $s>p$.
In this article we work with non-linear quantizers, which lead to better -mostly weakly optimal- results.  \end{remark}

The corresponding lower bound reads as follows.

\begin{theo}[Lower bound]\label{th516-1} There exist universal constants $c_1,c_2, c_3>0$ such that the following holds. For every L\'evy process $X$ with finite $F_2$, any $\eps>0$ with $F_1(\eps)\geq c_3$ one has
$$
D(c_1 F(\eps),1)\geq c_2 \,\eps.
$$
Moreover, if $\nu(\IR)=\infty$ or $\sig\not =0$, one has for any $s>0$,
$$
D(c_1 F_1(\eps),s)\gtrsim c_2 \,\eps
$$
as $\eps\dto 0$.
In the case where $F_2\equiv \infty$, one has $D(r,1)=\infty$ for any $r\ge 0$.
\end{theo}

\begin{remark}
So far one cannot replace $F_1$ by $F$ in the second statement of Theorem~\ref{th516-1}. Since mostly $F_1$ and $F$ are weakly equivalent when $\eps$ tends to zero, the second estimate typically  leads to sharp results. Nevertheless, it would be interesting to find out, whether one can close this remaining gap.
\end{remark}

Note that we have not specified the basis of the logarithm. However, all results stated above are valid for any basis. The choice of the basis has only an influence on the constants in the theorems. We will work with the basis $2$ when proving the upper bounds, since this seems more appropriate in the context of binary representations. When proving the lower bounds we switch to the natural logarithm.

\subsection{Examples} \label{sec:exa}
In this subsection, we apply the above results to some common L\'{e}vy processes.

\begin{example}[Stable L\'{e}vy process] Let us consider the case of an $\alpha$-stable L\'{e}vy process. Here we have $\nu(d x) = (C_1 \indi{x<0}+ C_2 \indi{x>0}) |x|^{-\alpha-1}\, d x$, and one can easily verify that $F_1(\ep)=C^1_\alpha \ep^{-\alpha}$ and $F_2(\ep)=C^2_{\alpha} \ep^{-\alpha}$. All assumptions of the main theorems are satisfied and we conclude that for all moments $s_1>0$, $s_2\in(0,\alpha)$ and all $p\geq 1$,
$$
D(r,s_1)\approx D^{(e)}(r,s_1)\approx D^{(q)}(r,s_2)\approx r^{-1/\alpha}.
$$
This improves results from \cite{creutzig2006} and \cite{LP}.
\end{example}

Note that the coding complexity $\alpha$-stable L\'{e}vy process is smaller than the one of a $2$-stable L\'{e}vy  process, i.e.\ Brownian motion. In fact, this is true for all L\'{e}vy process.

\begin{example}[L\'{e}vy process with non-vanishing Gaussian component]
It is easy to calculate that $F_i(\eps) \leq c \eps^{-2}$ for $i=1,2$. Therefore, if $\sigma\neq 0$ then $F(\eps)\approx F_1(\eps) \approx \eps^{-2}$.

This has two implications. Firstly, in presence of a Gaussian component, the coding complexity of the L\'{e}vy process is the same as for Brownian motion, as long as our results apply. In case $\sigma=0$, the coding complexity is weakly bounded from above by that of Brownian motion.

More precisely, $$D^{(e)}(r,s)\leq C r^{-1/2},\qquad\text{for any L\'{e}vy process,}$$
and $$D^{(e)}(r,s)\approx r^{-1/2},\qquad\text{if $\sigma\neq 0$.}$$
On the other hand, under the assumptions (a) and (b), $$D^{(q)}(r,s)\leq C r^{-1/2},\qquad\text{for any L\'{e}vy process,}$$
and, $$D^{(q)}(r,s)\approx r^{-1/2},\qquad\text{if $\sigma\neq 0$.}$$

In fact, by a modification of (\ref{eqn:concentr4}) one can show that (b) is not necessary if $\sigma\neq 0$.
\end{example}

\begin{example}[Gamma process] Let us consider a standard Gamma process. In this case, $\nu(d x) = \indi{x>0} x^{-1} e^{-x} d x$ and one gets $F_1(\eps)\approx \log 1/\eps$ and $F_2(\eps)\approx (\log 1/\eps)^2$. Consequently,  for fixed $p,s\in[1,\infty)$, there exist constants $c_1,c_2,c'_1,c_2'\iR_+$ such that for all $\eps \geq0$
$$D^{(e)}( c_1 (\log 1/\eps)^2, s) \leq c_2 \eps
$$
and
$$
D(c'_1 (\log 1/\eps)^2, s) \geq c_2' \eps.
$$
Therefore,
$$
D(r,s)=\exp\bigl(-e^{\cO(1)} \, \sqrt r\bigr)\qquad \text{and}\qquad D^{(e)}(r,s)=\exp\bigl(-e^{\cO(1)} \, \sqrt r\bigr).
$$
Note that Theorem~\ref{thm:upperquant} does not apply since condition (\ref{eqn:addcond}) fails to hold.
\end{example}

\begin{example}[Compound Poisson process] \label{exa:pp} Let $(N(t))_{t\geq 0}$ be a standard Poisson process. Let furthermore $Y, Y_1, Y_2, \ldots$ be i.i.d.\ random variables that are not a.s.\ equal to $0$  and independent of the Poisson process. Then
$$X(t) := \sum_{i=1}^{N(t)} Y_i
$$
is a compound Poisson process, i.e.\ a L\'{e}vy process with L\'{e}vy measure $\nu=\P_Y$ and drift $b=\E [Y \indi{|Y|\leq 1}]$.

It is immediately clear that $F_1(\eps)\leq 1$ and $F_2(\eps) \approx \E \left[\log\left( \frac{|Y|}{\eps}\right)\indi{|Y|\geq\eps}\right]$ so that $F_2$ dominates $F$ when $\eps$ is small. Thus the main complexity is induced  by the ``large jumps''. For fixed $p, s\in[1,\infty)$, the main theorems imply the existence of constants $c_1,c_2,c'_1,c_2'\iR_+$ such that
$$D^{(e)}\left(  c_1 \E \left[\log\left( \frac{|Y|}{\eps}\right)\indi{|Y|\geq\eps}\right], s\right)\leq c_2 \eps$$
and
$$
D\left(  c'_1 \E \left[\log\left( \frac{|Y|}{\eps}\right)\indi{|Y|\geq\eps}\right], s\right)\geq c'_2 \eps
$$
Hence,
$$
D(r,s)=\exp\bigl(-e^{\cO(1)} \,  r\bigr) \text{ and } D^{(e)}(r,s)=\exp\bigl(-e^{\cO(1)} \,  r\bigr).
$$
A more precise result for a subclass of compound Poisson processes was already obtained in the dissertation of Vormoor \cite{Vor07}. In particular, in those cases, the rates of quantization and entropy coding error differ.

Note that in the case of a compound Poisson processes we cannot use Theorem \ref{thm:upperquant} on the quantization error, since condition (b) is not satisfied.
%The reason can be understood very well by looking at the Poisson process: we need $\approx N \log 1/\eps$ bits to encode the jump points, where $N$ is the number of jumps of the process. In this case, our estimate (\ref{eqn:concentr4}) of the probability of the non-typical case fails.
\end{example}

\section{Upper Bounds} \label{sec:upper}
\subsection{An Explicit Coding Strategy}
In this subsection, we describe an explicit coding strategy that can be used to encode a L\'{e}vy process. We derive that the strategy has a mean error of order $\ep$ and that the bit complexity is given by the quantity in (\ref{eqn:complexityoverall}). In the following subsections we use this strategy in order to prove upper bounds for the entropy coding error and the quantization error.

The reconstruction $\hat{X}=\Theta_{\eps,b,m}(X)$ will be a step function with the step heights being integer multiples of $\ep$, i.e.\ we use an $\ep \IZ$ grid to approximate $X$. For this purpose, let us define $\gri$ to be a nearest neighbour projection of $\IR$ onto $\ep \IZ$. As a first step, we subtract the drift of the process by setting $X'(t):=X(t)-b(\eps) t$, where $b(\eps)$ is a drift compensation term given by $$b(\eps):=b-   \int_{[-1,1] \backslash [-\eps,\eps]} x \, \nu(d x)+ \int_{[-\eps,\eps] \backslash [-1,1]} x \, \nu(d x).$$

\paragraph{Notation.} Set $S_0:= 0$ and let
$$S_{i} := \inf\left\lbrace t>S_{i-1} : |X_t' - \gri(X'_{S_{i-1}})| \geq 2\ep \right\rbrace;\qquad i\iN,$$
 be the first exit time of the process $\left( X_s' - \gri(X_{S_{i-1}}')\right)_{s\geq S_{i-1}}$ from the interval $[-2\ep,2\ep]$.

Let $M:=\max\{ i : S_i< 1 \}$. Some of the stopping times $S_i$ are induced by jumps larger than $\eps$. These shall be called {\em large jumps}.

\paragraph{Coding procedure.} Note that it is possible to detect the jump points $(S_i)_{i=1,\dots,M}$ by a single swipe through the interval $[0,1]$.
%Note that $S_i<S_{i+1}$ a.s., by the definition of the stoping times $S_i$ and since $X'$ is $\ID[0,\infty)$-valued. We encode the jump points $S_i$, $i=1, \ldots, M$, by points $\hat{S}_i$, respectively, such that $S_i\leq \hat{S}_i< S_{i+1}$.
For each jump we encode its height and its time separately by using prefix-free representations:
we use a prefix-free representation for the integers $\Upsilon_1:\IZ\backslash\{0\}\to\{0,1\}^*$ (as outlined in Lemma~\ref{lem:code2}) to code the number $H_i/\eps\in\IZ$, where $H_i:=\gri(X_{S_i}')-\gri(X_{S_{i-1}}')$ denotes the discretised height.
Moreover, the time approximation $\hat S_i$ to $S_i$ is chosen in such a way that
\begin{align}\label{eq614-1}
S_i\leq \hat{S}_i< S_{i+1}\quad \text{ and }\quad\hat{S}_i-S_i\leq \ep^p/(|H_i|^p M).
\end{align}
For a visualization, cf.\ Figure~\ref{fig:jumps}. Concretely, we choose $\hat S_i$ as follows. By Lemma~\ref{lem:code1}, there is a coding scheme $\Upsilon_2 : \R\times \R_{>0} \to \{ 0,1 \}^*$, where, for $r\in[0,1]$, $\delta>0$, $\Upsilon_2(r,\delta)$ is the binary representation of a number $\overline{\Upsilon}_2(r,\delta) \in \bigcup_{n\ge0} 2^{-n} \IZ\cap [0,1]$ such that $\overline{\Upsilon}_2(r,\delta)\in [r,r+\delta]$.

We transmit the information in the following way: we divide the interval $[0,1)$ into $\lceil F_1(\eps)\rceil$ `boxes' (i.e.\ intervals) $I_j =[jF_1(\eps),(j+1) F_1(\eps)\wedge 1)$, $j=0,\dots,\lceil F_1(\eps)\rceil-1$. Each jump $S_i$ ($i=1,\dots, M$) is translated into the code
$$
\pi_i:= \text{`}0\text{'} *\Upsilon_1(H_i/\eps) * {\Upsilon}_2\left(F_1(\eps) S_i - \lfl F_1(\eps) S_i \rfl,F_1(\eps)\left( S_{i+1} - S_i\right)\wedge F_1(\eps)\ep^p/(|H_i|^p M)\right),
$$
where $*$ denotes the concatenation of strings. Note that $F_1(\eps) S_i - \lfl F_1(\eps) S_i \rfl$ is exactly the difference between the actual jump point and the left corner of the box, scaled on the unit interval. Then each block $j$ is described by the string
$$
\Pi_j:= \prod_{\{i: S_i\in I_j\}} \pi_i,
$$
and finally the complete information is encoded as
$$
\prod _{j=0}^{\lceil F_1(\eps)\rceil-1} \bigl(\Pi_j * \text{`}1\text{'} \bigr).
$$
It is easy to check that this provides indeed a prefix-free representation of $\bigl((\hat S_i,H_i)_{i=1,\dots,M},M\bigr)$, and the corresponding approximation defines a deterministic map $\Theta_{\eps,b(\eps),F_1(\eps)}$ by
 $$\hat X_t := \Theta_{\eps,b(\eps),F_1(\eps)}(X)(t) := b(\eps) t + \sum_{i=1}^M H_i \indi{\hat{S}_i \leq t},$$
 where $$\hat S_i := a_{\text{box},i} + \overline{\Upsilon}_2\left(F_1(\eps) S_i - \lfl F_1(\eps) S_i \rfl,F_1(\eps)\left( S_{i+1} - S_i\right)\wedge F_1(\eps)\ep^p/(|H_i|^p M)\right)$$ and $a_{\text{box},i}$ is the left corner of the box that contains $S_i$. Note that, in order to decode this value, it is sufficient to transmit a code for $F_1(\eps) S_i - \lfl F_1(\eps) S_i \rfl$. The chosen precision ensures (\ref{eq614-1}). Note that the parameters $\eps$, $b(\eps)$ and $F_1(\eps)$ describe the approximation scheme uniquely.

For convenience we will also consider the drift adjusted reconstruction $\hat X'$ defined by
$$
\hat X'_t:= \sum_{i=1}^M H_i \indi{\hat{S}_i \leq t}.
$$

 \begin{figure}
    \centering
        \epsfig{file=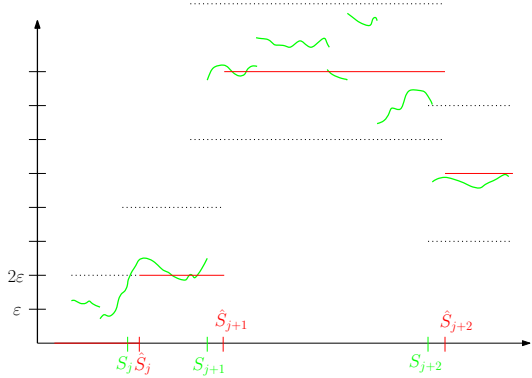,scale=0.4}
    \caption{The coding procedure}
    \label{fig:jumps}
\end{figure}

\paragraph{Waiting time for the jumps.} Let us estimate the waiting time for subsequent jumps. For this purpose, let $X^{(1)}$ be the process consisting of the (finitely many) jumps of $X'$ that are greater than $\eps$ and set $X^{(2)}:=X'-X^{(1)}$. Note that $X^{(2)}$ is a $(\left.\nu\right|_{[-\eps,\eps]},\sigma^2)$-L\'evy martingale. Denote by $\Gamma_1$ the stopping time induced by the first  jump of $X^{(1)}$. Note that  $|X_{S_{i-1}}'-g(X_{S_{i-1}}')|\le \eps/2$ a.s.\ so that due to the strong Markov property one has for all $t\geq 0$,
 \begin{eqnarray*} \pr{S_{i} - S_{i-1} \leq t~|~\cF_{S_{i-1}}} &\leq& \pr{ \sup_{0 < s \leq t} |X_s^{(2)}| > \frac32 \ep} + \pr{\Gamma_1 \leq t}\\
 &\leq& (3 \ep/2)^{-2} \E \sup_{0<s\leq t} |X_s^{(2)}|^2 + \nu([-\eps,\eps]^c) t\\
 &\leq&  \ep^{-2} \E | X_t^{(2)}|^2 + \nu([-\eps,\eps]^c) t,\end{eqnarray*}
 where the last step is justified by Doob's martingale inequality. By the compensation formula (\cite{bertoin}, p.\ 7) the last term equals
$F_1(\eps) t$.

Let $U_1, U_2, \ldots$ be a sequence of i.i.d.\ random variables. Then we have shown that for all jumps $S_i$, $$ \pr{S_{i} - S_{i-1} \leq t ~|~ \cF_{S_{i-1}}} \leq \max ( t F_1(\ep),1) = \pr{ U_i \leq F_1(\ep)t},$$ for all $t\geq0$ and $i\iN$. Consequently, we can couple the random times  $(S_{i} - S_{i-1})_{i\geq 1}$ with the sequence $(U_i)_{i\geq 1}$  such that  \begin{equation} F_1(\ep)( S_{i} - S_{i-1}) \geq  U_i. \label{eqn:coupling}\end{equation}

\paragraph{Coding error.} First, let us analyse the error of the approximation. With $\tilde X'=(g(X'_t))_{t\in[0,1]}$ one gets
\begin{align}
\|X-\hat X\|  & =  \|X'-\hat X'\| \le \underbrace{\|X'-\tilde  X'\|}_{\le 2\eps} + \|\tilde  X'- \hat X'\|. \label{eqn:lmodstar1}
\end{align}
Moreover, due to property (\ref{eq614-1})
\begin{equation}
\|\tilde  X'- \hat X'\|^p= \sum_{i=1}^M |H_i|^p (\hat S_i- S_i)\le \eps^p, \label{eqn:lmodstar2}
\end{equation}
so that $\|X-\hat X\|\le 3\eps$.

\paragraph{Coding complexity.} Let us count the number of bits needed in the approximation:
\begin{itemize}
\item Each change in a block is indicated by a '$1$' which gives in total $\lceil F_1(\eps) \rceil$ bits.
\item Each pair $(H_i,\hat S_i)$ is initialized by a '$0$' which gives in total $M$ bits.
\item Coding the numbers $H_1/\eps,\dots,H_M/\eps$ by using an appropriate representation $\Upsilon_1$ needs less than
$$
\sum_{i=1}^M 2\Bigl(2+\log\frac{|H_i|}\eps\Bigr)
$$
bits by Lemma \ref{lem:code2}.
\item Coding the numbers $\hat S_1,\dots,\hat S_M$ needs less than
$$
\sum_{i=1}^M 2\left(2+\log_+\frac{F_1(\eps)^{-1}} {\eps^p/(M|H_i|^p) \wedge (S_{i+1}-S_i)}\right)
$$
bits (see Lemma \ref{lem:code1}).
\end{itemize}
Therefore, the total  bit-length  is bounded from above by
$$
2\sum_{i=1}^M \left[ \log\frac{|H_i|}\eps + \log_+\frac{F_1(\eps)^{-1}} {\eps^p/(M|H_i|^p) \wedge (S_{i+1}-S_i)}    \right] +8 M+\lceil F_1(\eps) \rceil .
$$
This equals
$$
2\sum_{i=1}^M \left[ \log\frac{|H_i|}\eps + \log_+\left( \frac{M|H_i|^p} {F_1(\eps) \eps^p } \vee \frac{1} {F_1(\eps) (S_{i+1}-S_i)}\right)    \right] +8 M+\lceil F_1(\eps) \rceil.
$$
By (\ref{eqn:coupling}) and the inequality $\log_+ (x\vee y)\leq \log_+x + \log_+ y$, the latter is less than
\begin{align}\label{eq18-1}
2\sum_{i=1}^M \left[(1+p)\log_+\frac{|H_i|}{\eps} +\log_+ \frac{1}{U_i}\right] + 2 M \log_+\frac{M}{F_1(\eps)} + 8 M+\lceil F_1(\eps) \rceil.
\end{align}

% \begin{align} 2\sum_{i=1}^M \left[ (p+1) \log_+ \frac{|\Delta X_{S_i}|}\eps +\log \frac1{U_i}\right] +2M\log_+ \frac M{F_1(\eps)}+(10+4p)M+\lceil F_1(\eps) \rceil . \end{align}
Next, recall from (\ref{eqn:coupling}) that $F_1(\eps) (S_i-S_{i-1})\geq U_i$ so that
$$
F_1(\eps) M \geq \sum_{i=1}^M F_1(\eps) (S_i-S_{i-1}) \geq \sum_{i=1}^M U_i;
$$
and using the convexity of $\log_+(1/\cdot)$ one gets with Jensen's Inequality
$$
\sum_{i=1}^M \log \frac1{U_i}= M \sum_{i=1}^M \frac{1}{M} \log \frac{1}{U_i} \geq M \log_+\frac 1{\sum_{i=1}^M \frac {U_i}M}\geq M \log_+ \frac{M}{F_1(\eps)}.
$$
% Consequently we conclude with (\ref{eq18-1}) that \begin{align}\label{eq618-2} 2\sum_{i=1}^M \left[ (p+1) \log_+ \frac{|\Delta X_{S_i}|}\eps +2 \log \frac1{U_i}\right] +(10+4p)M+\lceil F_1(\eps) \rceil \end{align} is an upper bound for the bit-length.
We conclude with (\ref{eq18-1}) that $$ 2\sum_{i=1}^M \left[(1+p)\log_+\frac{|H_i|}{\eps} +2 \log \frac{1}{U_i}\right] + 8 M+\lceil F_1(\eps) \rceil$$
is an upper bound for the bit-length.

We conclude with (\ref{eq18-1}) that $$ 2\sum_{i=1}^M \left[(1+p)\log_+\frac{|H_i|}{\eps} +2 \log \frac{1}{U_i}\right] + 8 M+\lceil F_1(\eps) \rceil$$
is an upper bound for the bit-length. Denoting for any time $t>0$ the jump at time $t$ by $\Delta X_t=X_t-X_{t-}$ allows us to estimate
$|H_i|\leq |\Delta X_{S_i}|+\frac52 \eps$ so that basic analysis gives
$$
\log_+\frac {|H_i|}{\eps} \leq 5+ \log_+ \frac{|\Delta X_{S_i}|}{\eps}.
$$
Consequently, the bit-length is bounded by
\begin{align}
& 4 \sum_{i=1}^M  \log \frac{1}{U_i} + 2(1+p) \sum_{t\in(0,1]} \log_+ \frac{|\Delta X_{t}|}{\eps}+ (18+10p) M+\lceil F_1(\eps) \rceil \nonumber \\
& \leq K_1(p) \sum_{i=1}^M  [1+\log \frac{1}{U_i}] + K_2(p) \sum_{t\in(0,1]} \log_+ \frac{|\Delta X_{t}|}{\eps} + F_1(\eps)+1,
 \label{eqn:complexityoverall} \end{align}
where $K_1(p)$ and $K_2(p)$ are constants only depending on $p$.

\subsection{Proof of Theorem~\ref{thm:upperentropy}}
\begin{proof} By (\ref{eqn:lmodstar1}) and (\ref{eqn:lmodstar2}) the error (and thus the mean error, for all moments $s>0$) is less than $3 \ep$.

On the other hand, the coding complexity of the algorithm constructed above is given by (\ref{eqn:complexityoverall}). Let us look at what the different terms amount to on average. Note that
$$\E \sum_{t\in(0,1]} \log_+ \frac{|\Delta X_{t}|}{\eps}=F_2(\ep),$$ by the compensation formula (\cite{protter}, p.\ 29). Finally, by Lemma~\ref{lem:recursionunifo}, we have $$\E \sum_{i=1}^{M} (1+\log U_i^{-1}) \leq c F_1(\ep).$$

This shows that the expected bit length of the whole message is less than $c_1 F(\ep)$, with some constant $c_1$ depending only on $p$, as required. \end{proof}

\subsection{Proof of Theorem~\ref{thm:upperquant}}

\begin{proof} We use the coding scheme explained above. However, we encode by the zero function in case that
the number of small jumps, $M$, exceeds $C_1 F(\eps)$, where $C_1$ is a constant to be chosen presently. The same is done if the complexity to encode the jump heights of the large jumps, namely $\sum_{t\in(0,1]} \log_+ |\Delta X_{t}|/\eps$, or the complexity to encode the positions of the jumps, namely $\sum_{i=1}^{M} (1+\log U_i^{-1})$, is larger than $C_2 F(\eps)$, where $C_2$ is a constant to be chosen presently. Let us define $\TT$ to be the event that none of the above cases occurs, i.e.\ the `typical case'.

Note that, by the exponential compensation formula (\cite{bertoin}, p. 8),\begin{multline}
\pr{\sum_{t\in(0,1]} \log_+ \frac{|\Delta X_{t}|}{\ep}  > C_2 F(\eps)} \leq e^{-C_2 \mu F(\eps)} \E e^{\mu \sum_{t\in(0,1]} \log_+ \frac{|\Delta X_{S_i}|}{\ep}}\\
\leq e^{-C_2 \mu F(\eps)} e^{ -\int_{|x|\geq \eps} 1 - \left(|x|/\ep\right)^\mu \, \nu (d x)}
\leq e^{-C_2 \mu F(\eps)} e^{E F(\eps)}
\leq e^{-C_2/2 \mu F(\eps)}, \label{eqn:concentr4}
\end{multline}
where $E$ is some constant depending on the finite constant in (\ref{eqn:addcond}) only. The last step holds for $C_2$ large enough. On the other hand, by the Chebyshev Inequality, $$\pr{\sum_{i=1}^{C_1 F(\eps)} (1+\log U_i^{-1}) > C_2 F(\eps)} \leq e^{-C_2/2\, F(\ep)},$$ for $C_2$ large enough. Finally, one proves, e.g.\ using the same discretization as in (\ref{eqn:discrtrick}), that for $C_1$ large enough, $$\pr{M > C_1 F(\ep) } \leq e^{-C_1/2\, F(\ep)}.$$ Therefore, for some positive constant $C$ depending on $\mu$ and $E$, we have $\pr{\TT^c} \leq \exp(-C F(\eps))$.

Let $r>0$ be chosen by $1/q+1/r=1/s$. Let $\kappa>0$ be chosen small enough such that $C \nu([-\kappa,\kappa]^c) \geq r$. This is possible, since $\nu([-\kappa,\kappa]^c)$ tends to infinity when $\kappa\to 0$, by condition~(b). Then, for $\eps<\kappa$, $$F(\eps)\geq F_2(\eps)=\int_{[-\eps,\eps]^c} \log \frac{|x|}{\eps} \nu(d x) \geq \nu([-\kappa,\kappa]^c) \log \frac{1}{\eps} \geq - \frac{1}{C} \log \eps^r.$$ Thus, \begin{equation} \pr{\TT^c} \leq e^{-C\, F(\eps)} \leq \eps^r. \label{eqn:removecond2}\end{equation}

Note that the bit complexity of our algorithm is constant if $\TT^c$ occurs and, by (\ref{eqn:complexityoverall}), less than $C F(\ep)$ if $\TT$ occurs, where $C$ depends on $\mu$ and $E$. Then we have for the mean error, using the H\"{o}lder Inequality and $s\geq 1$,
\begin{eqnarray*} \left( \E \norm{X-\hat{X}}^s \right)^{1/s} &\leq &
\left( \E \ind_{\TT} \norm{X-\hat{X}}^s \right)^{1/s} + \left( \E \ind_{\TT^c} \norm{X-\hat{X}}^s \right)^{1/s}
\\
&\leq &  c_2 \ep + \left( \E \ind_{\TT^c}^r \right)^{1/r} \left(\E \norm{X}^q \right)^{1/q}
\\
&\leq & c_2 \ep \left[ 1 + \ep^{-1} c_{2}^{-1} \pr{\TT^c}^{1/r} \left(\E \norm{X}^q \right)^{1/q} \right],
\end{eqnarray*}
where the term in brackets is bounded, by assumption (a) and (\ref{eqn:removecond2}). Note that the argument works analogously for $0<s<1$.  \end{proof}

\begin{remark} \label{rem:conda} It is easy to see that condition (a) is equivalent to the condition $$\int_{|x|>1} |x|^q\, \nu( d x ) < \infty.$$ \end{remark}

\begin{remark} \label{rem:condbc} Let us assume that (a) holds. A sufficient condition for (b) to hold is that $\nu([-2\eps,2\eps]^c)\leq c\cdot \nu([-\eps,\eps]^c)$ for some $0<c<1$ and all $0<\eps\leq \eps_0$. This can be seen as follows:
\begin{multline*} \int_{\eps<|x|\leq \eps_0} \left( \frac{|x|}{\eps}\right)^\mu \, \nu(d x) \leq \sum_{k=0}^{\log (\eps_0/\eps)} \int_{2^k \eps < |x| \leq 2^{k+1} \eps} \left(  \frac{|x|}{\eps} \right)^\mu \, \nu(d x)
\\
\leq \sum_{k=0}^{\log (\eps_0/\eps)} \nu([-2^k\eps,2^k\eps]^c) 2^{(k+1)\mu}
\leq \sum_{k=0}^\infty c^{k} 2^{(k+1)\mu} \nu([-\eps,\eps]^c).\end{multline*}
Choosing $0<\mu<(-\log c)\wedge q$ yields $$\int_{|x|> \eps} \left( \frac{|x|}{\eps}\right)^\mu \, \nu(d x)\leq K(\mu,c) \, \nu([-\eps,\eps]^c) + \eps^{-\mu} \int_{\eps_0<|x|\leq 1}  |x|^\mu \, \nu(d x)+ \eps^{-\mu} \int_{|x|>1}  |x|^q \, \nu(d x),$$ which implies (\ref{eqn:addcond}).

Note that, in particular, this is the case if $\eps\mapsto\nu([-\eps,\eps]^c)$ is regularly varying at zero with negative exponent.
\end{remark}

\subsection{Technical tools}
In this section, we prove some technical tools that are needed in the proofs of the main results.
\begin{lemma} Let $\lambda>0$  and let $(U_i)_{i\geq 1}$ be an i.i.d.\ sequence of random variables uniformly distributed in $[0,1]$. For $N:= \min\lbrace n\iN : \sum_{i=1}^n U_i \geq  \lambda \rbrace$ one has $$\E \sum_{i=1}^N (1+\log U_i^{-1}) \leq 6\lceil 2 \lambda\rceil .$$  \label{lem:recursionunifo} \end{lemma}

\begin{proof} Let $0\leq s\leq 1$.
Define $N(s):= \min\lbrace n\iN_0 : \sum_{i=1}^n U_i \geq s  \rbrace$ and consider the function
$$\Psi(s) := \E \sum_{i=1}^{N(s)} (1+\log U_i^{-1}).$$
We are interested in $\Psi(\lam)$. Clearly, $\Psi(s)=0$ for $s\leq0$ and $\Psi$ is increasing. Moreover, one has for $s>0$,
\begin{align*} \Psi(s) & = \int \left( 1+ \log x^{-1}
 + \E \sum_{i=1}^{N(s-x)} (1+ \log U_i^{-1}) \right) \, d \P_{U_1}(x)\\
 & = 1+\int_0^1 - \log x\, d x + \int \E \sum_{i=1}^{N(s-x)}(1+ \log U_i^{-1} ) \, d \P_{U_1}(x)\\
 & = 1+\log e + \int \Psi( s - x)\, d \P_{U_1}(x) \leq  3 + \int \Psi( s - x )\, d \P_{U_1}(x).
\end{align*}
Let us define \begin{equation} U_1' := \begin{cases} 0 & U_1 \leq 1/2\\ 1/2 &U_1> 1/2.\end{cases} \label{eqn:discrtrick} \end{equation} Then $U_1'\leq U_1$; and since $\Psi$ is increasing, we have
$$\Psi(s) \leq 3 + \int  \Psi(s-x) \, \P_{U_1'}(x)= 3 + \frac{1}{2}\,  \Psi(s) + \frac{1}{2}\,  \Psi\left(s-\frac{1}{2}\right).$$
Therefore, $\Psi(s) \leq 6 +  \Psi\left(s-\frac{1}{2 }\right)$ and we get that $$\Psi(\lam) \leq 6 +  \Psi\left(\lam-\frac{1}{2}\right) \leq  6 + 6+ \Psi\left(\lam-1\right)\leq \ldots \leq  6\cdot \lceil 2 \lambda\rceil.$$
 \end{proof}

Let us finally gather two facts concerning the coding of integers and real numbers from a given interval, respectively.

\begin{lemma} There is a universal coding scheme that returns a prefix free code $\Upsilon_1(x)\in\{0,1\}^*$ for a given integer $x\in \IZ$ that has a length of at most $2(2+\log x)$ bits. \label{lem:code2}
\end{lemma}

\begin{proof} The sign is encoded by a first bit. Thus, assume $x>0$, because $x=0$ can be encoded by $\text{`}00\text{'}$. Let $n:=\min\lbrace l\in\N ~|~ x<2^l\rbrace$. Then
$2^{n-1}\leq x<2^n$. Consider the representation of $x$ in the
binary system. Because of the definition of $n$, this representation
must have $n$ bits, the first one of which is a `1'.

A prefix free code for $x$ is given by $n$ times `1', followed by a
`0' and the $n-1$ bit long representation of $x$ in the binary
system having taken away the redundant leading `1'.

The length of the code is $2n+1$, which is less than $2 (1+\log_+ x)$.
\end{proof}

Let us remark that Lemma~\ref{lem:code2} can be improved up to the order $\log x + C \log \log x + D$, as shown
in \cite{elias1975}.

\begin{lemma} There exists a universal coding strategy $\Upsilon_2 : \R\times\R_{>0} \to \{ 0,1\}^*$ such that, for any $\delta>0$ and $r\in [0,1]$, $\Upsilon_2$ returns the prefix free binary representation $\Upsilon_2(r,\delta)$ of a number $\overline{\Upsilon}_2(r,\delta)\in [r,r+1]$ with $r\leq \overline{\Upsilon}_2(r,\delta)\leq r+\delta$ that needs at most $2(2-\log \delta)$ bits. \label{lem:code1}
\end{lemma}

\begin{proof} Let $N:=\min\lbrace n ~:~ \delta \geq 2^{-n-1}\rbrace$. We choose $\overline{\Upsilon}_2(1,r,\delta)\in [r,r+1]\cap S_N$ nearest possible, but larger than $r$, where $$S_N:= \bigcup_{n=0}^N 2^{-n} \IZ.$$ This ensures that $0\leq \overline{\Upsilon}_2(1,r,\delta)  - r\leq
2^{-(N+1)}\leq \delta$, as required.

Any number $\hat{r}\in [0,1]\cap S_N$ has a unique representation $\hat{r}=k 2^{-n}$,
with $k$ uneven, $1\leq k\leq 2^n-1$, $1\leq n\leq N$. As a prefix
free code $\Upsilon_2(1,r,\delta)$ for $\overline{\Upsilon}_2(1,r,\delta)$ we chose the prefix free code for the
integer $2^{n-1}+(k+1)/2$. Since $\overline{\Upsilon}_2(\gamma,\delta)\in S_N$, we have to encode
integers from $2$ up to at most $2^N$, which, by Lemma~\ref{lem:code2}, requires at most $2(1+N)$ bits, which is less
than $2(1-\log\delta)$ bits, by the definition of $N$. \end{proof}

\section{Lower bound} \label{sec:lowerbound}

The aim of this section is to provide lower bounds for the distortion rate function of the L\'{e}vy process.
The analysis is divided into three subsections. First we introduce some concepts of information theory and we prove some preliminary results. Next, we provide a lower bound based on $F_2$. In the last subsection we give a lower bound in terms of $F_1$. Both lower bounds then immediately imply Theorem \ref{th516-1}.

So far $p$ is a fixed value in $[1,\infty)$. Since the distortion rate function is increasing in the parameter $p$, we can and will fix $p=1$ in the following discussion.

As mentioned before, we can freely choose the basis of the logarithm in the proof of the main theorems. For the rest of this article, we fix as basis $e$.

\subsection{Preliminaries}

First we  will introduce some concepts of information theory.
We will need the concept of conditional mutual information. Let $A,B$ and $C$ denote random vectors attaining values in some Borel spaces. Then one defines the mutual information between $A$ and $B$ given $C$ as
$$
I(A;B|C)=\int I(A;B|C=c)\,d\IP_C(c),
$$
where
$$
I(A;B|C=c)=\begin{cases}\int \log \frac{d\IP_{A,B|C=c}}{d\IP_{A|C=c}\otimes\IP_{B|C=c}} \,d\IP_{A,B|C=c}&\text{ if } \IP_{A,B|C=c}\ll \IP_{A|C=c}\otimes\IP_{B|C=c}\\ \infty &\text{ otherwise.}\end{cases}
$$
A summary of computation rules for the mutual information can be found in \cite{Iha93}.

\begin{lemma}\label{le515-1}
For $n\iN$, let $Y_0,\dots,Y_{n-1}$ and $\hat Y_0,\dots,\hat Y_{n-1}$ and $H$ denote random variables in possibly different Borel spaces. We write shortly $Y=(Y_0,\dots,Y_{n-1})$, $Y^i=(Y_0,\dots, Y_i)$ for $0\leq i\leq n-1$ and $\hat Y=(\hat Y_0,\dots,\hat Y_{n-1})$. Then one has
$$
I(Y,H;\hat Y) \geq I(Y_0; \hat Y_0|H)+ I(Y_1;\hat Y_1|H,Y^0) +\dots +I(Y_{n-1};\hat Y_{n-1}|H, Y^{n-2}).
$$
\end{lemma}

Moreover, we will need to evaluate the distortion rate function for other originals than the L\'evy process $X$ and for other distortions than $L_p[0,1]$-norm. For a measure $\mu$ on a Borel space~$E$ and a measurable function $\rho:E\times E\to[0,\infty]$ (\emph{distortion measure}) we write
$$
D(r|\mu,\rho)=\inf\bigl\{ \IE[\rho(X,\hat X)] : \hat X \ E\text{-valued r.v.\ with } I(X;\hat X)\leq r\bigr\}.
$$
Moreover, we associate to a map $\rho:E\to [0,\infty]$  the \emph{difference distortion measure}
$\rho:E\times E\to[0,\infty]$ (denoted by the same identifier) given as $\rho(x,\hat x)=\rho(x-\hat x)$.
Sometimes we will also consider  a general moment $s>0$ and write
$$
D(r|\mu,\rho,s)=\inf\bigl\{ \IE[\rho(X,\hat X)^s]^{1/s} : \hat X \ E\text{-valued r.v.\ with } I(X;\hat X)\leq r\bigr\}.
$$
Moreover, we will omit $\rho$ if it is the norm based distortion induced by the $L^1[0,1]$-norm.

The following proposition allows us  to separately consider the influence of the large jumps and the diffusive part with small jumps onto the coding complexity of the L\'evy process:

\begin{propo}\label{propo1t}
Let $E$ be a Borel-space and assume that $(E,+)$ is an Abelian group
such that the sum is Borel-measurable. Denote by $A$ and $B$
independent $E$-valued random elements and suppose that there exists a measurable map $\vphi:E\to E^2$ with
\begin{align}\label{eq16t}
\vphi(A+B)=(A,B)\text{ a.s.}
\end{align}
Then, under any difference distortion measure $\rho$ on $E$, one has for every $r\geq0$:
$$
D(r|\IP_{A+B},\rho)\geq D(r|\IP_A,\rho).
$$
\end{propo}

\begin{proof}
Fix $r\geq0$. Next, we use that the distortion rate function  $D(\cdot|\IP_A,\rho)$ is convex. We denote by $f$ a
tangent of $D(\cdot|\IP_A,\rho)$ at the point $r$. Then, for any
random element $Z$ on $E$,
\begin{align*}
\IE[\rho(A+B,Z)]&= \int \IE[\rho(A, Z-b)|B=b] \,dP_B(b)\\
&\geq \int f(I(A;Z|B=b))\,dP_B(b)= f\Bigl( \int
I(A;Z|B=b)\,dP_B(b)\Bigr)\\
&=f(I(A;Z|B)).
\end{align*}
Therefore,
$$
\inf_{\{Z:I(A;Z|B)\leq r\}} \IE[\rho(A+B,Z)]\geq f(r)=D(r|\IP_A,\rho).
$$
On the other hand, by assumption (\ref{eq16t}), $I(A+B;Z)=I((A,B);Z)$ for any random
element $Z$ on $E$. Hence,
$$
I(A+B;Z)=I((A,B);Z) = I(B;Z)+ I(A;Z|B)\geq I(A;Z|B).
$$
Therefore,
\begin{align*}
D(r|\IP_{A+B},\rho)&=\inf_{\{Z: I(A+B;Z)\leq r\}} \IE[\rho(A+B,Z)]\\
&\geq \inftwo{\{Z\text{ r.v.\,on }E:}{I(A;Z|B)\leq r\}} \IE[\rho(A+B,Z)]\geq D(r|\IP_A,\rho).
\end{align*}
\end{proof}

\subsection{Lower bound based on \texorpdfstring{$F_2$}{F2}}

\begin{theo}\label{th1218-1}
There exists some universal constant $c$ such that for all $\eps>0$,
$$
D\Bigl(\frac{\kappa(\eps)} e \, F_2(\eps)\Big|X,L_1[0,1],1\Bigr)\geq c\,\kappa(\eps)\, \eps,
$$
where $\kappa(\eps)=\kappa(\eps,\nu)=\lfl \nu([-\eps,\eps]^c)\rfl/  \nu([-\eps,\eps]^c)$.
\end{theo}

The proof of the theorem is based on the following idea: in order to find an approximation of accuracy $\eps$, one needs to allocate about $\log_+ |X_{t}-X_{t-}|/\eps$  bits (nats) for each big jump.

The problem is related to a minimization problem that we want to introduce now. Let $\Pi$ be a finite non-negative measure on a measurable space $(E,\cE)$ and let $h:E\to [0,\infty)$ denote a Borel-measurable function with
$$
\int \log_+ h(x) \,d\Pi(x)<\infty.
$$
The aim is now to minimize for given $r>0$ the target function
$$
\int  h(x) \exp(- \xi(x))\,\Pi(dx)
$$
over all measurable functions $\xi:E\to [0,\infty)$ satisfying the constraint
\begin{align}\label{constr}
\int \xi(x)\,d\Pi(x) \leq r.
\end{align}

\begin{lemma}\label{le0116-1}
Assuming that $\{h>0\}$ has not $\Pi$-measure zero, the minimization problem
possesses a $\Pi$-a.e.\ unique solution of the form
\begin{align}\label{eq0116-3}
\xi(x)= \log_+ \frac{h(x)}{\lam},
\end{align}
where $\lam=\lam(r)>0$ is an appropriate parameter depending on $r>0$. When the optimal function~$\xi$ is as in (\ref{eq0116-3}),   then the  minimal value of the target function is
$$
\int \lam\wedge h(x) \,\Pi(dx).
$$
\end{lemma}

\begin{proof} The proof is based on a Lagrangian analysis.
Let $\zeta(y)=\exp(- y)$ ($y\in[0,\infty)$) and consider its convex conjugate
$$
\bar \zeta(z)=\inf_{y\geq0} [\zeta(y)+yz]\qquad(z\geq 0).%=\frac zp \log_+ \frac zp +\frac zp \wedge 1.
$$
Let $\lam>0$ and  denote by $\tilde \Pi$  the $\sigma$-finite measure with $\frac{d\tilde\Pi }{d \Pi} (x)=h(x)$.
Now observe that for a non-negative function $\xi$ satisfying the constraint (\ref{constr}) one has
\begin{align}
\int h(x) \exp(-\xi(x)) d\Pi(x)& \geq\int \Bigl[\zeta(\xi(x)) + \lam \frac{\xi(x)}{h(x)}\Bigr] d\tilde\Pi(x) -\lam r \label{eq0426-1}\\
&\geq \int \bar \zeta\left(\frac\lam{h(x)}\right) d\tilde\Pi(x) -\lam r.\label{eq0426-2}
\end{align}
The last expression in this estimate does not depend on the choice of $\xi$. If we can establish equality in the above estimates for certain $\xi$ and $\lam$, then this $\xi$  minimizes the problem.

Next, we note that one has equality in (\ref{eq0426-1}) iff
\begin{align}\label{eq0116-1}\begin{cases}
\int \xi(x) \,d\Pi(x)=r \text{ and}\\
\xi(x)=0 \text{ for }\Pi\text{-a.e. } x \text{  with }h(x)=0.\end{cases}
\end{align}
We need to look for a non-negative function $\xi$ and a parameter $\lam>0$ such that (\ref{eq0116-1}) is valid and such that
\begin{align}\label{eq0116-2}
\bar \zeta\left(\frac\lam{h(x)}\right)= \zeta(\xi(x)) + \frac{\lam }{h(x)} \xi(x) \ \ \text{for }\tilde\Pi \text{-a.e. }x.
\end{align}
It is straightforward to verify that for positive $z$ the function
$$
[0,\infty)\ni y \mapsto \zeta(y)+zy\in(0,\infty)
$$
attains its unique minimum in $y=\log_+\frac 1z$. Therefore, condition (\ref{eq0116-2}) is equivalent to
$$
\xi(x)= \xi_\lam(x):= \log_+\frac{h(x)}{\lam} \ \ \text{for }\tilde\Pi \text{-a.e. }x.
$$
Together with (\ref{eq0116-1}) a sufficient criterion for $\xi$ being a minimum is the existence of a $\lam>0$ such that
$$\begin{cases}
\int \xi(x) \,d\Pi(x)=r \text{ and}\\
\xi(x)=\xi_\lam(x) \text{ for }\Pi\text{-a.e. } x.\end{cases}
$$
Such a $\lam$ exists since the function
$$
g:(0,\infty)\ni \lam \mapsto \int \xi_\lam (x)\,dx \in[0,\infty)
$$
is continuous (due to the dominated convergence theorem) and
satisfies
$$ \lim_{\lam\dto0} g(\lam)=\infty \qquad\text{and}\qquad\lim_{\lam\to \infty} g(\lam)=0.$$
Note that if $\xi$ does not coincide with $\xi_\lam$ $\Pi$-a.e. (where $\lam$ is such that $g(\lam)=r$), then one of the inequalities (\ref{eq0426-1}) or (\ref{eq0426-2}) is a strict inequality so that $\xi$ does not minimize the target function.
\end{proof}

\begin{proof}[ of Theorem \ref{th1218-1}] Fix $\eps>0$.
Due to  Proposition \ref{propo1t} we can assume without loss of generality that $X$ is a pure jump process with jumps bigger than $\eps$.
Next, let $l=1/\nu([-\eps,\eps]^c)$, $n=\lfl 1/l\rfl$ and
$$
r=\frac {nl}e \int_{[-\eps,\eps]^c} \log \frac{|x|}{\eps} \,\nu(dx)=\frac{\kappa(\eps)}{e}\,F_2(\eps).
$$
We will prove that for an arbitrarily fixed reconstruction $\hat X$ with $I(X;\hat X)\leq r$ one has
$$
\IE [ \|X-\hat X\|_{L_1[0,1]}] \geq c nl\, \eps,
$$
where $c>0$ is a universal constant.

We let
$$
\pi:L_1[0,1] \to \ell_1^n,\qquad (x_t) \mapsto \Bigl(\Bigl|\int_{il}^{(i+1)l} (2 \ind_{\{t\geq (2i+1)l/2\}}-1) x_t \,\frac{dt}l\Bigr|\Bigr)_{i=0,\dots,n-1}
$$
and consider
$$
Y=(Y_i)_{i=0,\dots,n-1}=\pi( X) \qquad\text{and}\qquad \hat Y=\pi(\hat X).
$$
The map $\pi$ is  $l^{-1}$-Lipschitz continuous so that
\begin{align}\label{eq0117-2}
\IE[\|Y-\hat Y\|_{\ell_1^n}] \leq l^{-1}  \IE[\|X-\hat X\|].
\end{align}
Moreover, $\pi$ is invariant under uniform shifts on each time interval $[i/n,(i+1)/n)$ so that in particular,
$$
\pi(X)=\pi\Bigl(X-\sum_{i=0}^{n-1} X_{\frac{2i+1}2l} \ind_{[il,(i+1)l)}\Bigr).
$$
Due to the strong Markov property of the L\'evy process, the random variables $Y_0,\dots,Y_{n-1}$ are i.i.d.
We shall derive a lower bound for $\IE[\|Y-\hat Y\|_{\ell_1^n}]$.

For $i=0,\dots,n-1$ consider the events
$$
A_i=\{X\text{ contains in }[il,(i+1)l)\text{ exactly one jump}\}.
$$
and the random vector $H=(H_i)_{i=0,\dots,n-1}$ given by
$$
H_i=\begin{cases} \text{size of the jump in }[il,(i+1)l) & \text{ if }A_i\text{ occurs},\\
0&\text{ otherwise}.
    \end{cases}
$$
%and denote by $T_i$ the time of the unique jump if $A_i$ occurs and $-1$ otherwise. Moreover, let $H_i=\Delta  X_{T_i}$ with the convention $\Delta X_{-1}=0$.
Next, denote $Y^i=(Y_0,\dots,Y_i)$ for $i=0,\dots,n-1$ and $Y^{-1}=0$. Our objective is to find a lower bound for
\begin{align}\label{eq0426-3}
\IE[\|Y-\hat Y\|_{\ell_1^n}]\geq \IE \Bigl[\sum_{i=0}^{n-1} \IE\bigl[ |\ind_{A_i} Y_i - \ind_{A_i} \hat Y_i| \big| H, Y^{i-1}\bigr]\Bigr].
\end{align}
For each  $i\in\{0,\dots,n-1\}$ we analyze the inner expectation.
Let $f_i(h,y^{i-1})=I(Y_i,\hat Y_i|H=h, Y^{i-1}= y^{i-1})$ and consider the random variable
$$
R_i=f_i(H, Y^{i-1}).
$$
Given $H$ and $Y^{i-1}$, the r.v.\ $Y_i$ is uniformly distributed on $[\ind_{\{H_i<0\}} H_i/2, \ind_{\{H_i>0 \}} H_i/2]$.
Therefore,
\begin{align}\label{eq0426-4}
\IE\bigl[ |\ind_{A_i} Y_i - \ind_{A_i} \hat Y_i~|~ \big| H, Y^{i-1}\bigr]\geq D(R_i|\cU[0,|H_i|/2],|\cdot|),
\end{align}
where $\cU[0,u]$ denotes the uniform distribution on $[0,u]$. Now there exists a universal constant $c>0$ such that for any $\bar r\geq0$ and any $u\geq0$
$$
D(\bar r|\cU[0,u/2],|\cdot|) \geq c\,u\, e^{- \bar r}.
$$
Together with (\ref{eq0426-3}) and (\ref{eq0426-4}) we arrive at
$$
\IE[\|Y-\hat Y\|_{\ell_1^n}]\geq c \IE \sum_{i=0}^{n-1} |H_i| \,e^{- R_i}.
$$
With $\Pi$ defined as the product measure $\IP\otimes \sum_{j=0}^{n-1} \delta _j$ we get
\begin{align}\label{eq0426-5}
\IE[\|Y-\hat Y\|_{\ell_1^n}]\geq c \int |H_i| \,e^{- R_i}\, d\Pi (\om, i).
\end{align}
On the other hand,  one has $\IE[R_i]=I(Y_i, \hat Y_i| H, Y^{i-1})$ by definition so that by Lemma \ref{le515-1}
$$
\int R_i \,d\Pi(\om,i) =\sum_{i=0}^{n-1} \IE[R_i] \leq I(Y,H;\hat Y) \leq I(X;\hat X)\leq r.
$$

Now consider the minimization problem for the target function
$$
\int |H_i| \,e^{- R_i}\, d\Pi (\om, i),
$$

where the minimum is taken over all random variables $R_i$ ($i=0,\dots,n-1$) satisfying  $\int R_i \,d\Pi(\om,i)\leq r$.
The law of $H_i$ is $(1-e^{-1}) \delta_0+ \frac{1}{e \,\nu([-\eps,\eps]^c)} \nu|_{[-\eps,\eps]^c}$ so that
$$
\int \log_+ \frac {|H_i|}{\eps} \,d\Pi(\om,i)= \frac ne \int_{[-\eps,\eps]^c}  \log \frac {|x|}{\eps} \frac {\nu(d x)}{\nu([-\eps,\eps]^c)}=r.
$$
Hence, Lemma \ref{le0116-1} implies that the optimal value in the minimization problem is
$$
\int \eps\, \ind_{\{h_j\not=0\}}  \, d\Pi(\om,j)=\frac 1e  n\, \eps.
$$
Together with (\ref{eq0117-2}) and (\ref{eq0426-5}) we get that
$$
\IE[\|X-\hat X\|]\geq  \frac{c}{e}\,  ln \eps
$$
which yields the assertion.

\end{proof}

%
%
%
% $$
% D(\bar r |(h_i U_i/2)_{i=0,\dots, n-1},\rho_n)\geq c \, \zeta(\bar r, (h_i)),
% $$
% where $\zeta$ is defined as
% $$
% \zeta(\bar r,(h_i))=\min_{(\bar r_i)} \sum_{i=0}^{n-1} h_i^p e^{-p \bar r_i}
% $$
% and the minimum is taken among all non-negative sequences $(\bar r_i)$ with $\sum_{i=0}^{n-1} \bar r_i=\bar r$.
%
% Denote by $R_i((h_i))$ ($i=0,\dots,n-1$) the values of the $\zeta$-minimizers $\bar r_i$ when taking $\bar r=R((h_i))$.
% Together with  (\ref{eq1218-1})  we obtain
% \begin{align*}
% \IE[\rho_n(Y,\hat Y)]&\geq c \int  \sum_{j=0}^{n-1} h_j^p e^{-p R_j((h_i))} \, d\IP_{(H_i)}((h_i)_{i=0,\dots,n-1}).
% %\min_{(\bar r_i)}\Bigl(\sum_{i=0}^{n-1} h_i^p e^{-p \bar r_i}\Bigr)\, d\IP_{(H_i)}((h_i)_{i=0,\dots,n-1}),
% \end{align*}
% Next, let $\Pi=\IP_{(H_i)}\otimes (\sum_{i=0}^{n-1} \delta_i)$ and observe that one can rewrite the previous formula as
% \begin{align}\label{eq1218-2}
% \IE[\rho_n(Y,\hat Y)]\geq c \int   h_j^p e^{-p  R_j((h_i))} \, d\Pi((h_i)_{i=0,\dots,n-1},j).
% \end{align}
% On the other hand
% $$
% \int R_j((h_i)) \,d\Pi((h_i),j)\leq r.
% $$
% Moreover, the law of $H_i$ has distribution function $(1-e^{-1})d\delta_0+ e^{-1} \ind_{(h,\infty)} \,\nu([-h,h]^c)^{-1} d(\nu+\nu^*)$ so that
% $$
% \int \log_+ \frac {h_j}{h} \,d\Pi((h_i),j)= \frac ne \int_{[-h,h]^c}  \log \frac {|x|}{h} \frac {d\nu(x)}{\nu([-h,h]^c)}=r.
% $$
% Consequently, the previous lemma combined with (\ref{eq1218-2}) gives
% $$
% \IE[\rho_n(Y,\hat Y)]\geq c \int h^p \ind_{\{h_j\not=0\}}  \, d\Pi((h_i)_{i=0,\dots,n-1},j)=\frac ce  nh^p.
% $$

\subsection{Lower bound related to the \texorpdfstring{$F_1$}{F1}-term}

\begin{theo}\label{th0117-1}
There exist positive universal constants $c_1$ and $c_2$ such that the following statements are true.
For any $\eps>0$ with $F_1(\eps)\geq 18$, one has
$$
D\bigl(c_1\,F_1(\eps) ,1\bigr)\geq c_2\,\eps.
$$
If $\nu(\IR)=\infty$ or $\sig\not =0$, then  for any $s>0$, one has
$$
D\bigl(c_1\,F_1(\eps) ,s\bigr)\gtrsim c_2\,\eps
$$
as $\eps\dto0$.
\end{theo}

Let us give some heuristics on the proof of the theorem. As we have mentioned before the drift adjusted process $X'$ needs approximately the time $1/F_1(\eps)$ to leave an interval of length $2\eps$. Assuming that the process is symmetric the process leaves the stripe to either of the sides with equal probability (here one also needs to assume that one starts in the center of the interval). Thus in order to have a coding of accuracy $\eps$ one needs to describe at least in which direction the process left the stripe for most of the exits. This requires about  $F_1(\eps)$ bits.

As the following remark explains, it suffices to prove the theorem for symmetric L\'evy processes.

\begin{remark}
Let $X^*$ denote an independent copy of $X$ and observe that for  $s\in(0,1]$
$$
D(2r~|~\IP_{X-X^*},s)\leq 2^{1/s} D(r~|~\IP_X,s).
$$
The process $X-X^*$ is a symmetric L\'evy process and the functions describing its complexity are
$$
\tilde F_1(\eps)=2 F_1(\eps)\text{ and } \tilde F_2(\eps)=2 F_2(\eps).
$$
\end{remark}

We assume from now on that the L\'evy process $X$ has no drift and a symmetric L\'evy measure~$\nu$.

\begin{lemma}\label{le0115-3}
Let $\eps>0$ and denote
$$
T=\inf \{t\geq0 : |X_t|\geq \eps\}.
$$
Then
$$
\IP(T\geq t) \leq \frac9{4F_1(2\eps)\,t}.
$$
\end{lemma}

\begin{proof}
We consider  a L\'evy process $X^*$ with L\'evy measure $\nu^*=\nu\circ \pi^{-1}$ with $\pi:\IR\to[-2\eps,2\eps]$ being the projection onto the interval $[-2\eps,2\eps]$. Then the exit times $T$ and
$$
T^*=\inf\{t\geq0 : |X^*_t|\geq \eps\}
$$
are equal in law. Moreover, the process $X^*_{T^*\wedge \cdot}$ is a by $3\eps$ uniformly bounded martingale and the quadratic variation process $[X^*]$ of $X^*$  is a subordinator with Doob-Meyer Decomposition
$$
[X^*]_t=\bigl([X^*]_t- 4\eps^2 F_1(2\eps)\,t\bigr)+ 4\eps^2 F_1(2\eps) \,t.
$$
Therefore,
\begin{align*}
9\eps^2 \geq \IE (X_{T^*}^2)&= \lim_{t\to\infty} \IE (X_{t\wedge T^*}^2)\\
&=\lim_{t\to\infty} \IE [X]_{t\wedge T^*} = 4\eps^2 F_1(2\eps) \lim_{t\to\infty} \IE(t\wedge T^*)= 4\eps^2 F_1(2\eps)  \,\IE(T^*).
\end{align*}
Consequently,
$$
\IE T^*\leq \frac9{4F_1(2\eps)}
$$
and the assertion follows immediately.
\end{proof}

\begin{lemma}\label{le0115-2}
Let $Y$ be a Bernoulli r.v. Then for $d\in[0,1/2]$
$$
D(d \log 2d+ (1-d)\log 2(1-d)~|~Y,\rho_\mathrm{Ham})\geq d ,
$$
where $\rho_\mathrm{Ham}$ denotes the Hamming distance.
\end{lemma}

\begin{proof}
Interpret $Y$ as a random variable attaining values in the group $\IZ_2$ consisting of two elements. Then $\rho$ can be interpreted as a difference distortion measure on $\IZ_2$, that means for $x,\hat x\in\IZ_2$
$$\rho(x,\hat x)=\rho(x-\hat x):= \ind_{\{x-\hat x=0\}}.$$
Next, note that  for $d\in[0,1/2]$:
$$
\phi(d):=\sup\{ H(Z):  Z \  \IZ_2\text{-valued}, \IE[\rho(Z)]\leq d\} = -d \log d- (1-d)\log (1-d).
$$
We use the concept of the Shannon lower bound to finish the proof: Let $\hat Y$ denote a $\IZ_2$-valued reconstruction  with
$\IE[\rho(Y,\hat Y)]=d\leq 1/2$; then
\begin{align*}
I(Y;\hat Y)&=H(Y)-H(\hat Y|Y)=H(Y)-H(\hat Y-Y|Y)\geq H(Y)-H(\hat Y-Y) \\
& \geq \log 2-\phi(d)=d \log 2d+ (1-d)\log 2(1-d).
\end{align*}
\end{proof}

In the proof we will use that for the Bernoulli distribution $\mu^\text{Ber}$ and Hamming distortion $\rho^{\text{Ham}}$ one has for any $d\in[0,1/2]$ that
$$
D(d \log 2d+ (1-d)\log 2(1-d)~|~\mu^\text{Ber},\rho^\mathrm{Ham})= d .
$$

The proof of the lower bound is based on a comparison with a simpler distortion rate function.
For $q\in[0,1/2]$ let $\mu_q$ denote the measure that assigns probabilities $q$ to $\pm1$ and $1-2q$ to $0$. Moreover denote by
$\mu^{\otimes n}_q$ its product measure, consider the distortion measure
$$
\rho(x,\hat x)= \ind_{\{x \cdot \hat x=-1\}} \qquad(x\in\{\pm1,0\}, \hat x\in\{\pm1\})
$$
and denote
$$\rho_n(x,\hat x)=\sum_{i=0}^{n-1} \rho(x_i,\hat x_i).
$$
As reconstruction we allow any $\{\pm1\}^n$-valued random vector.

\begin{propo}\label{prop0427-1}
For any $r\geq 0$, $n\iN$ and any L\'evy process with symmetric L\'evy measure, one has
$$
D(r|\IP_X,s) \geq \frac \eps {4n^{1/p}} \,D(r|\mu_{q}^{\otimes n}, \rho_{n},s).
$$
where $$q=\frac1{8} \Bigl(1-\frac9{F_1(2\eps)\, l}\Bigr)\vee 0 .$$
\end{propo}

\begin{proof}
First fix $n\iN$, $r\geq0$ and a reconstruction $\hat X$ with $I(X;\hat X)\leq r$. We denote $l=1/n$ and consider again
$$
\pi:L_1[0,1] \to \ell_1^n, \qquad(x_t) \mapsto \Bigl(\Bigl|\int_{il}^{(i+1)l} (2 \ind_{\{t\geq (2i+1)l/2\}}-1) x_t \,\frac{dt}l\Bigr|\Bigr)_{i=0,\dots,n-1}.
$$
The map $\pi$ is $l^{-1}$-Lipschitz continuous and the random vector
$$
Y:=(Y_i)_{i=0,\dots,n-1}=\pi( X)
$$
consists of i.i.d.\ entries. Additionally, we set $\hat Y=(\hat Y_i)_{i=0,\dots,n-1}=\pi(\hat X)$.
Next, consider random vectors $Z=(Z_i)_{i=0,\dots,n-1}$ and $\hat Z=(\hat Z_i)_{i=0,\dots,n-1}$ defined as
$$
Z_i= \begin{cases} \sgn (Y_i)  & \text{ if } |Y_i|\geq \eps/4\\
0 & \text{ otherwise}
     \end{cases} \ \ \text{ and }  \ \ \hat Z_i=\begin{cases} 1  & \text{ if } \hat Y_i\geq0\\
-1 & \text{ otherwise}.\end{cases}
$$
Recalling the Lipschitz continuity of $\pi$  we get that
$$
\|X-\hat X\| \geq l  \|Y-\hat Y\|_{\ell_1^n} \geq l \frac\eps4 \sum_{i=0}^{n-1} \rho(Z_i,\hat Z_i).
$$
Therefore,
$$
\IE[\|X-\hat X\|^s]^{1/s}\geq \frac{\eps}{4n}\IE[\rho_n(Z,\hat Z)^{s}]^{1/s}.
$$
Certainly, $Z$ is distributed according to $\mu^{\otimes n}_q$, where $q=\IP(Y_1\geq \eps/4)$. Since $I(X;\hat X)\geq I(Z;\hat Z)$ we obtain that in general
$$
D(r|\IP_X,s) \geq \frac \eps {4n^{1/p}} \,D(r|\mu_{q}^{\otimes n}, \rho_{n},s).
$$

Next, we show that $D(r|\mu_{q}^{\otimes n}, \rho_{n},s)$ is increasing in $q$. Indeed, let $0\leq q<q'\leq 1/2$, let $Z$ denote an $\mu_{q'}^{\otimes n}$ distributed r.v., and let $\hat Z$ denote a reconstruction for $Z$ with $I(Z;\hat Z)\leq r$. Moreover, let   $A=(A_0,\dots,A_{n-1})$ be a random vector consisting of i.i.d.\ Bernoulli random variables with success probability $q/q'$ that are independent of $Z$ and $\hat Z$ (for finding such a sequence one might need to enlarge the probability space), and set $\tilde Z:=(\tilde Z_i)_{i=0,\dots,n-1}:=(A_i Z_i)_{i=0,\dots,n-1}$.  Then $\tilde Z$ is $\mu_{q}^{\otimes n}$-distributed and one has
$$
\IE[\rho_n(\tilde Z,\hat Z)]\leq \IE[\rho_n(Z,\hat Z)] \ \text{ and } \ I(\tilde Z;\hat Z)\leq I(A, Z;\hat Z) =I(Z;\hat Z).
$$

It remains to prove that $\IP(Y_i\geq \eps/4)\geq  \frac1{8} \bigl(1-\frac9{F_1(2\eps)\, l}\bigr)$.
We fix $i\in\{0,\dots,n-1\}$ and let
$$
(\tilde X_t)_{t\in[-l/2,l/2)}=(X_{t+\frac{2i+1}2 l}-X_{\frac{2i+1}2 l})_{t\in[-l/2,l/2)}.
$$
The processes $(\tilde X_t)_{t\in[0,l/2)}$ and $(-\tilde X_{-t})_{t\in[0,l/2]}$ are independent L\'evy martingales with L\'evy measure $\nu$. Denote $T^+=\inf\{t\geq0: \tilde X_t\geq \eps\text{ or }t\geq l/2\}$ and observe that
\begin{align*}
\IP\bigl(Y_i\geq \frac \eps4\bigr) &\geq \IP\Bigl(-\int_{0}^{l/2} \tilde X_{-t}\,dt\geq 0 , T\leq l/4, \int_{T}^{l/2} [\tilde X_t-\tilde X_T]\,dt\geq0\Bigr)\\
&= \IP\Bigl(\int_{0}^{l/2} \tilde X_{-t}\,dt\leq 0\Bigr)\, \IP(T\leq l/4)\, \IP\Bigl(\int_{T}^{l/2} [\tilde X_t-\tilde X_T]\,dt\geq0|T\leq l/4\Bigr)\\
&=\frac14\, \IP(T^+\leq l/4).
\end{align*}
Set $T=\inf\{t\geq0: |\tilde X_t|\geq \eps$ or $t\geq l/2\}$. Then the symmetry of $\nu$ together with Lemma \ref{le0115-3} implies that
$$
\IP(T^+\leq l/4)\geq \frac12 \IP(T\leq l/4) \geq \frac12 \Bigl(1-\frac9{F_1(2\eps)\, l}\Bigr)
$$
so that
$$
\IP\bigl(Y_i\geq \frac \eps4\bigr) \geq\frac1{8} \Bigl(1-\frac9{F_1(2\eps)\, l}\Bigr).
$$
\end{proof}

\begin{lemma}\label{le0427-1}
Let $\mu^\mathrm{Ber}$ and $\rho^{\mathrm{Ham}}$ denote the Bernoulli distribution and the Hamming distance, respectively. Then
$$
D(r|\mu_q,\rho)\geq 2q\, D\Bigl(\frac r{2q} \Big|\mu^{\mathrm{Ber}},\rho^\mathrm{Ham}\Bigr).
$$
\end{lemma}

\begin{proof}
Let $X$ denote a $\mu_q$ distributed r.v.\ and let $\hat X$ denote a $\{\pm1\}$-valued reconstruction with $I(X;\hat X)\leq r$.
Denote $f(\bar x)= I(X;\hat X\big | |X|=\bar x)$ for $\bar x\in\{0,1\}$ and let
$$
\bar r= f(1) \ \text{ and  } \ R=f(|X|).
$$
Then one has $\IE R=I(X;\hat X| |X|) \leq I(X;\hat X) \leq r$ so that due to the non-negativity of $R$
$$
\bar r\leq \frac{r}{\IP(|X|=1)}= \frac r{2q}.
$$
Next, we write
$$
\IE \rho(X,\hat X)= \IE\Bigl[ \ind_{\{X\not=0\}} \IE[ \ind_{\{X\not= \hat X\}}\big| |X|]\Bigr]
$$
and note that conditional on $|X|=1$, $X$ is a Rademacher random variable so that
$$
\IE \rho(X,\hat X)\geq \IP(|X|=1) \, D(\bar r|\mu^{\mathrm{Ber}},\rho^\mathrm{Ham}).
$$
Together with the above estimate for $\bar r$ this completes the proof.
\end{proof}

\begin{proof}[ of Theorem \ref{th0117-1}, $\text{1}^{\text{st}}$ statement]
Let $\eps>0$ with $F_1(2\eps)\geq 18$ and choose $n\iN$ maximal with
$n\leq F_1(2\eps)/18$. Then
$$
q:=\frac18\Bigl(1-\frac{9n}{F_1(2\eps)}\Bigr)\vee 0\geq \frac1{16}.
$$
Additionally, there exists a universal constant $C_3>0$ such that $n\geq C_3 F_1(2\eps)$.
Next, we shall apply Proposition \ref{prop0427-1}.
We fix $r_0<\log 2$ arbitrarily and   set $r=\frac18 nr_0$. Then $r\geq C_1\, F_1(2\eps)$ for some constant $C_1$ only depending on the choice of $r_0$. Thus with Proposition \ref{prop0427-1} one gets
\begin{align}\label{eq515-2}
D(C_1 F_1(2\eps),s)\geq D(r,s)\geq \frac{\eps}{4 n} D\Bigl(\frac18 nr_0\Big|\mu_q^{\otimes n},\rho_n,s\Bigr).
\end{align}

Recall that statement 1 of the theorem considers the case where  $s=1$. But $D\Bigl(\frac18 nr_0|\mu_q^{\otimes n},\rho_n\Bigr)$ is a distortion rate function for a single letter distortion measure and an i.i.d.\ original, and, therefore,
$$
D\Bigl(\frac18 nr_0\Big|\mu_q^{\otimes n},\rho_n\Bigr)= n\, D\Bigl(\frac18 r_0\Big|\mu_q,\rho\Bigr)
$$
The latter distortion rate function has been related to that of a Bernoulli variable in Lemma~\ref{le0427-1}:
$$
D\Bigl(\frac18 nr_0\Big|\mu_q^{\otimes n},\rho_n\Bigr)\geq n\, 2q \, D\Bigl(\frac1{16q}  r_0\Big|\mu^\mathrm{Ber},\rho^\mathrm{Ham}\Bigr).
$$
Since $q\geq 1/16$ the rate in the last distortion rate function is bounded by $r_0<\log 2$ so that the distorion rate function yields a value $C_4>0$ strictly bigger $0$. Altogether,
$$
D(C_1 F_1(2\eps),1)\geq  \frac \eps 2 q C_4 \geq C_2\, 2\eps,
$$
where $C_2=\frac18({C_4}/8)^{1/p}$. Switching from $2\eps$ to $\eps$   finishes the proof of the first assertion.
\end{proof}

The proof of the second statement relies on the following concentration property:

\begin{lemma}\label{le515-2}
Let $\rho:\IR\times \IR \to [0,\infty]$ be a measurable function, let $(U_i)_{i\iN}$ be a sequence of  independent bounded random
variables, and denote by $U^{(n)}$ the random vector $(U_i)_{i=1,\dots,n}$. Supposing that there exists $u^*\iR$ such that
\begin{align}\label{eq515-3}
\IE[\rho(U_1,u^*)^2]<\infty,
\end{align}
one has for any $s>0$ and $r>0$:
$$
\liminf_{n\to\infty} \frac1n \,D(nr|U^{(n)}, \rho_n,s)\ge d,
$$
where $d=D(r|U_1,\rho,1)$ and $\rho_n$ is the single letter distortion measure belonging to $\rho$.
\end{lemma}

As one can see in the proof the moment condition (\ref{eq515-3}) can be easily relaxed. Similar ideas  are used   in \cite{Der06b} to prove concentration of the approximation error.

\begin{proof} Without loss of generality we assume that $D(r|U_1,\rho)>0$. Our moment condition implies that $D(\cdot|U_1,\rho)$ is finite, convex and continuous on $[0,\infty)$.
 Following the standard proof of Shannon's source coding theorem, there is a family of codebooks $(\cC(n))_{n\iN}$ such that
\begin{itemize}
\item $\{(u^*,\dots,u^*)\}\subset \cC(n)\subset \IR^n$,
\item  $\log |\cC(n)|\lesssim nr$,
\item $\lim_{n\to\infty}\IP(\cT(n))=1$ for $\cT(n)=\{\min_{\hat u^{(n)}\in \cC(n)} \rho_n(U^{(n)},\hat u^{(n)})< (1+\eps(n)) d\}$ and an appropriate zero-sequence $(\eps(n))_{n\iN}$.
\end{itemize}
% Then the Cauchy-Schwarz inequality implies that for $s\in[1,3/2]$
% \begin{align*}
% \IE[ (\min_{\hat u^{(n)}\in \cC(n)} \rho_n(U^{(n)},\hat u^{(n)}))^s]^{1/s}& \leq (1+\eps(n)) d +\IE[\ind_{\cT^c} \rho_n(U^{(n)},(u^*,\dots,u^*))^s]^{1/s}\\
% &\leq (1+\eps(n)) d + \IP(\cT^c)^{1/2} \IE[\rho(U_1,u^*)^{2s}]^{1/2s} \\
% &\to d.
% \end{align*}

For any $n\iN$, let $\hat U^{(n,1)}$ denote an arbitrary reconstruction for
$U^{(n)}$ such that we have $I(U^{(n)},\hat U^{(n,1)})\leq nr$, and let $\hat
U^{(n,2)}=\arg\min_{\hat u^{(n)}\in\cC(n)} \rho_n(U^{(n)},\hat u^{(n)})$. We
fix $\eta\in(0,1)$ arbitrarily and choose
$$
J=\begin{cases} 1 &\text{ if }\log\frac{d\IP_{U^{(n)},\hat
U^{(n,1)}}}{d\IP_{U^{(n)}}\otimes \IP_{\hat U^{(n,1)}}}\leq nr\text{ and }
\rho_n(U^{(n)},\hat U^{(n,1)})\leq (1-\eta)d,\\
2 &\text{ else,}
\end{cases}
$$
and $\hat U^{(n)}=\hat U^{(n,J)}$.

Next, we will use that
$$
I(U^{(n)}; \hat U^{(n)})\leq I(U^{(n)}; \hat U^{(n)}, J) = \inf_Q H(\IP_{U^{(n)}, \hat U^{(n)}, J} \| \IP_{U^{(n)}}\otimes Q),
$$
where the infimum is taken over all probability measures $Q$ on $\IR\times \{1,2\}$ and $H$ denotes the relative entropy. We choose
$$Q=\frac12\bigl[\IP_{\hat U^{(n,1)}}\otimes \delta_1 + Q^*\otimes \delta_2\bigr] \ \ \text{ with } \ \  Q^*=\frac1{|\cC(n)|}\sum_{\hat u^{(n)}\in \cC(n)} \delta_{\hat u^{(n)}}$$
in order to get an appropriate bound for $I(U^{(n)}; \hat U^{(n)})$:
\begin{align*}
I(U^{(n)},\hat U^{(n)}) &\leq H(\IP_{U^{(n)}, \hat U^{(n)}, J} \| \IP_{U^{(n)}}\otimes Q)\\
&= \int \log \frac{d\IP_{U^{(n)},\hat
U^{(n)}, J}}{d\IP_{U^{(n)}}\otimes Q}  d\IP_{U^{(n)},\hat U^{(n)},J}\\
&\leq \int_{\{J=1\}} \log \frac{d\IP_{U^{(n)},\hat
U^{(n)}, J}}{d\IP_{U^{(n)}}\otimes \IP_{\hat U^{(n,1)}}\otimes \delta_1} d\IP_{U^{(n)},\hat
U^{(n)},J} \\
& \ \ \ +\int_{\{J=2\}} \log \frac{d\IP_{U^{(n)},\hat
U^{(n)},J}}{d\IP_{U^{(n)}}\otimes Q^*\otimes \delta_2}  d\IP_{U^{(n)},\hat U^{(n)},J} +\log 2
\end{align*}
Note that the measures $\IP_{U^{(n)},\hat U^{(n)}, J}$ and $\IP_{U^{(n)},\hat U^{(n,1)}, J}$ agree on the set $\{J=1\}$ so that by the construction of $J$ one has $\log \frac{d\IP_{U^{(n)},\hat
U^{(n)}, J}}{d\IP_{U^{(n)}}\otimes \IP_{\hat U^{(n,1)}}\otimes \delta_1}\leq nr$ on $\{J=1\}$. Moreover, one has $\log \frac{d\IP_{U^{(n)},\hat
U^{(n)}, J}}{d\IP_{U^{(n)}}\otimes Q^* \otimes \delta_2}\leq \log |\cC(n)|$ on $\{J=2\}$.
 Consequently, we can continue with
\begin{align*}
I(U^{(n)},\hat U^{(n)}) & \leq P(J=1) \,nr + P(J=2) \log |\cC(n)| +\log 2 \lesssim nr.
\end{align*}

On the other hand, basic transformations and the Cauchy-Schwarz Inequality yield
\begin{align*}
& \IE[\rho_n(U^{(n)},\hat U^{(n)})] \\ &= \IE[\ind_{\{J=1\}} \rho_n(U^{(n)},\hat
U^{(n,1)})] + \IE[\ind_{\{J=2\}} \rho_n(U^{(n)},\hat U^{(n,2)})]\\
&\leq (1-\eta) d \IP(J=1) + \IP(J=2) (1+\eps(n))d +\IP(\cT^c)^{1/2}
\IE[\rho_n(U^{(n)},(u^*,\dots,u^*))^{2}]^{1/2}\\
& \sim [(1-\eta) \IP(J=1)+\IP(J=2)]d.
\end{align*} Therefore, $\lim_{n\to\infty} \IP(J=1)=0$. Consequently, we arrive at
$$
\IE[\rho_n(U^{(n)},\hat U^{(n,1)})^s]^{1/s}\geq \IP(J\not =1)^{1/s} (1-\eta)
d\to (1-\eta) d
$$
and recalling that $\eta\in(0,1)$ was arbitrary finishes the proof.
\end{proof}

\begin{proof}[ of Theorem \ref{th0117-1}, $\text{2}^{\text{nd}}$ statement]
We define $r_0$, $q$ and $n$ as in the proof of the first statement.
By assumption $\nu(\IR)=\infty$ or $\sig\not=0$. Consequently, one has $\lim_{\eps\dto0} F_1(\eps)=\infty$ and $n$ converges to $\infty$ as $\eps$ tends to $0$.

We recall estimate (\ref{eq515-2}):
$$D(C_1 F_1(2\eps),s)\geq D(r,s)\geq \frac{\eps}{4 n} D\Bigl(\frac18 nr_0\Big|\mu_q^{\otimes n},\rho_n,s\Bigr).
$$
Now we conclude with Lemma \ref{le515-2} that
$$
D\Bigl(\frac18 nr_0\Big|\mu_q^{\otimes n},\rho_n,s\Bigr)\gtrsim n D\Bigl(\frac18 r_0\Big|\mu_q,\rho\Bigr).
$$
The assertion follows along the lines of the proof of the first statement.
\end{proof}

\subsection*{Acknowledgement} The research of the first-mentioned author was
supported by the DFG Research Center ``{\sc Matheon} -- Mathematics for key
technologies'' in Berlin.

\hypertarget{vd}{~}\pdfbookmark[1]{References}{vd}
\bibliographystyle{plain}%{abbrv}%{alpha}

\end{document}